\documentclass[a4paper,11pt]{amsart}
\allowdisplaybreaks
\usepackage{geometry}
\usepackage{amssymb,amsfonts}
\usepackage[all,arc]{xy}
\usepackage{enumerate}
\usepackage{mathrsfs}
\usepackage{changepage}
\usepackage{amsmath}
\usepackage{datetime}
\usepackage{faktor}
\usepackage{breqn}
\usepackage{centernot}
\usepackage{enumitem}
\usepackage{amsthm}
\usepackage{exercise}
\usepackage[stable]{footmisc}
\usepackage{graphicx}
\usepackage{subcaption}
\usepackage[normalem]{ulem}
\usepackage{float}
\usepackage[export]{adjustbox}
\usepackage{booktabs}
\usepackage{graphicx}
\usepackage{soul}
\usepackage{chngcntr}
\usepackage[
backend=biber,
style=alphabetic,
sorting=ynt
]{biblatex}
\counterwithin{table}{subsection}
\counterwithin{figure}{subsection}

\newcommand{\vs}{\vspace{2mm}}
\newcommand{\mods}{\ (\text{mod } n)}

\newtheorem{thm}{Theorem}[subsection]
\newtheorem{cor}[thm]{Corollary}

\newtheorem{lem}[thm]{Lemma}

\newtheorem{quest}[thm]{Question}

\theoremstyle{definition}
\newtheorem{defn}[thm]{Definition}
\newtheorem{defns}[thm]{Definitions}

\newtheorem{exmp}[thm]{Example}

\theoremstyle{remark}

\newenvironment{abs}[1]{%

\begin{quotation}\textit{Abstract}: #1 \end{quotation}
}{%
\vspace{1cm}
}

\begin{document}

\begin{center}
\huge{An Exploration of the Symmetry Groups of Certain Configurations of Points}\\
\vspace{2mm}
\large{Luke Boyer and Nick Payne}\\
\vspace{1mm}
\large{July 6, 2021}\\
\vspace{5mm}
\end{center}
\begin{adjustwidth}{20pt}{20pt}
\abs{
We start by introducing the basics of configurations of points and lines, and then move into discussing symmetry groups of these configurations.  Specifically, we explore how we might classify the symmetries of $(9_3)$ and $(10_3)$ geometric configurations, given the graph automorphisms of their underlying set-configurations.  Finally, we show how a specific class of combinatorial configurations called generalized cyclic configurations can be explored using this terminology, and give several interesting geometric results.  
}
\end{adjustwidth}

\tableofcontents
\section{Introduction}
Configurations are a collection of points and lines where each point is incident with the same number of lines (and vice versa). A fundamental question in the study of configurations is \emph{geometric realizability}. That is to say, given an abstract incidence structure (referred to as a set-configuration), do we know when we are able to draw (realize) a set of points and lines with the same incidence structure? Specifically, do we know when a structure of a set-configuration is exhibited in its geometric realizations? The answer to this question lies in understanding the symmetric structures of set-configurations. As one might expect, the set of all symmetries for a given configuration form a group which acts on the set of points (and lines) of a configuration.  Towards the beginning of this paper, we attempt to classify subgroups of some historically significant set-configurations and observe when geometric realizations with visually apparent symmetries of the same nature exist. Later on, in the second topic, we explore a specific type of set-configuration whose structure is cyclic in nature.  As it turns out, whether these configurations are geometrically realizable is depended on several parameters, even something as straightforward as the number of incidences per line.  In addition, we provide some commentary on how these realizations were done in the past, as well as provide our own novel observations and results.
\subsection{Acknowledgements}\hfill\\
We would like to thank our incredible mentor Ian Dumais for the countless hours he spent with us reviewing material and providing insight, and would like to thank Shane Calle for his thoughtful discussion.  Additionally, we are amazingly grateful of Vance Blankers and Matej Penciak for organizing and hosting this program, as well as for being there to answer any question we might have. \\
We are grateful for the RTG grant Algebraic Geometry and Representation Theory at Northeastern University,  DMS–1645877, which supported this research experience. We also thank the Northeastern Department of Mathematics and the Northeastern College of Science for additional support.
\hfill\\
\section{Basic Definitions and Concepts\footnote{Unless otherwise noted, definitions and theorems in this section are from \cite{Grunbaum}}}\hfill\\
While there are many different definitions of ``configuration", a configuration of points and lines can be defined as follows:
\begin{defn} \label{def:config}
     A \textbf{configuration} $(p_q,n_k)$  is a family of points and lines such that, for $p,q,n,k \in \mathbb{N}$:
    \begin{itemize}
    \item Each of the \textit{p} points is incident with precisely \textit{q} of the \textit{n} lines.
    \item Each of these \textit{n} lines is incident with precisely \textit{k} of the \textit{p} points.
\end{itemize} 
\end{defn}

For the remainder of this paper, we limit ourselves to $(n_3)$ configurations, i.e. configurations that have $n$ lines incident with $3$ points, and $n$ points incident with $3$ lines.  Additionally, it is important to note that, depending on context, the usage of ``points" and ``lines" in Definition 
\ref{def:config} is not necessarily accurate, which leads to the following set of definitions.
\begin{defns} Configurations can be classified as one of three types:
\begin{itemize}
    \item A \textbf{Geometric Configuration} is an image of points and lines in the (extended) Euclidean plane.
    \item A \textbf{Topological Configuration} is an image of points and pseudo-lines in the (extended) Euclidean plane.
    \item A \textbf{Combinatorial Configuration} is an abstract representation listing the incidences of ``points" for each ``line". In this case, points can be referred to as \textbf{marks}, and lines referred to as \textbf{blocks}.
\end{itemize}
\end{defns}
As one might expect, geometric configurations are also topological configurations, and topological configurations are also combinatorial configuration.  If a geometric or topological configuration is formed from a combinatorial configuration, we say that it is \textbf{underlied by}, and a \textbf{realization of}, the combinatorial configuration.

\begin{figure}[H]
\subfloat{\resizebox{.5\linewidth}{!}{
    \begin{tabular}{|c c c c c c c c c|}
    \hline
    1 & 1 & 1 & 2 & 2 & 2 & 3 & 3 & 3\\ 
    \hline
    4 & 5 & 6 & 4 & 5 & 6 & 4 & 5 & 6\\
    \hline
    7 & 8 & 9 & 8 & 9 & 7 & 9 & 7 & 8\\
    \hline
    \end{tabular}}}
\subfloat{\includegraphics[width=0.4\textwidth]{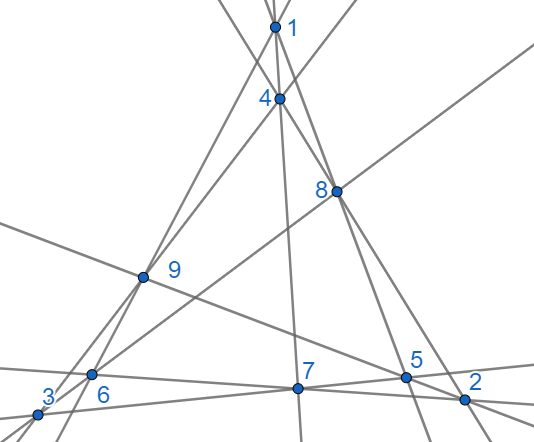}}
\caption{On the left, a $(9_3)$ combinatorial configuration.  On the right, its geometric realization.}
\end{figure}

\begin{defn}
$C'$ and $C''$ are considered \textbf{isomorphic} (sometimes also called \textbf{combinatorially equivalent}, or of the same \textbf{combinatorial type}) if the points and lines admit labels such that a one-to-one correspondence $\tau$ of points to points (and lines to lines) preserves incidences.  
\end{defn}

\begin{figure}[H]
\includegraphics[width=0.8\textwidth]{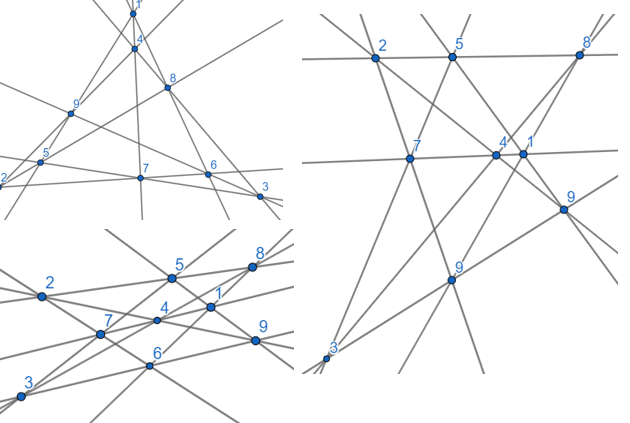}
{\caption{3 isomorphic configurations.}}
\end{figure}
\newpage
\section{Symmetries of Configurations}\hfill\\
\begin{figure}[H]
\subfloat{\includegraphics[width=0.4\textwidth]{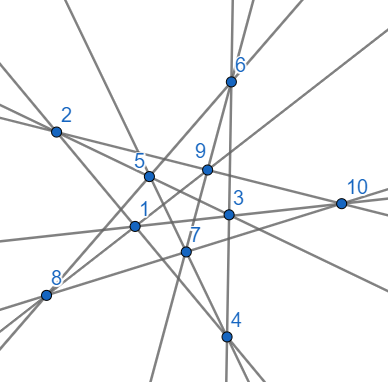}}
\subfloat{\includegraphics[width=0.4\textwidth]{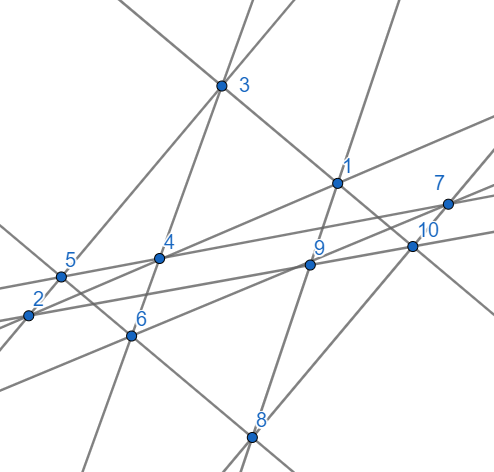}}\\

\subfloat{\resizebox{.5\linewidth}{!}{
    \begin{tabular}{|c c c c c c c c c c |}
    \hline
    10 & 1 & 2 & 3 & 4 & 5 & 6 & 7 & 8 & 9\\ 
    \hline
    1 & 2 & 3 & 4 & 5 & 6 & 7 & 8 & 9 & 10\\
    \hline
    3 & 4 & 5 & 6 & 7 & 8 & 9 & 10 & 1 & 2\\
    \hline
    \end{tabular}}}
\caption{A $(10_3)$ configuration with two isomorphic realizations.}
\label{fig:iso1}
\end{figure}

We begin with an example of a ($10_3$) configuration with two of its isomorphic realizations and their underlying set-configuration as shown in Figure \ref{fig:iso1}.  Notice that these realizations are symmetric, the first having 5-fold rotations symmetry, and the other having reflectional symmetry. One might wonder if the set-configuration exhibits properties that give rise to realizations with symmetries. This is exactly the motivation of the following section; we wish to be able to classify the symmetric structure of the set-configuration and determine when realizations exist with the same symmetries. From here on, we reserve the word \emph{symmetry} solely for geometric configurations, and refer to ``symmetries" of a set-configuration as \emph{automorphisms}. 

\subsection{Symmetries and Automorphisms}\hfill\\

First, we introduce some necessary definitions and theorems.
\begin{defn}
A \textbf{symmetry} of geometric configurations in the Euclidean plane is an isometry of the plane that map the configuration onto itself \cite{Grunbaum}.
\end{defn}


\begin{thm}[\cite{Grunbaum}]
The set of all symmetries of a geometric configuration form a group. If $G$ is an arbitrary geometric configuration we denote its symmetric group as $\text{Sym}(G)$ .
\end{thm}


\begin{defn}
An \textbf{automorphism} is a 1-to-1 incidence preserving transformation from points to points (and lines to lines) of a set-configuration.
\end{defn}

\begin{thm}\label{thm:auto}
The set of all automorphisms of a given set-configuration form a group subgroup of $S_n \times S_n$. If $C$ is an arbitrary set-configuration we denote its automorphism group as $\text{Aut}(C)$.
\end{thm}

\begin{proof}

Let $H$ be the set of all automorphisms of an $(n_k)$-configuration. Any element in $H$ must be a permutation of the $n$ lines paired within a permutation of $n$ points and thus is contained within $S_n \times S_n$. 

Consider the arbitrary automorphism $(\sigma_l, \sigma_p) \in H$ where each $\sigma$ is an incidence preserving permutation of the lines or points. We know permutations are cyclic, so let $o_l, o_p$ be the order of $\sigma_l, \sigma_p$, respectively. We have,
\[(\sigma_l, \sigma_p)^{o_lo_p} = e \implies (\sigma_l, \sigma_p)^{o_lo_p - 1} = (\sigma_l, \sigma_p)^{-1}\]

Since $(\sigma_l, \sigma_p)^{-1}$ is just the composition of $o_lo_p - 1$ many incident preserving mapping, it must also preserve incidences. Thus $(\sigma_l, \sigma_p)^{-1} \in H$.
The identity transformation preserves incidences and thus $H$ is a subgroup of $S_n \times S_n$.
\end{proof}

An immediate consequence of Theorem \ref{thm:auto} is the following:
\begin{cor}
For any geometric realization $G$ of a set-configuration $C$, $\text{Sym}(G) \leq \text{Aut}(C)$. \qedsymbol
\end{cor}

Now that we understand the distinction in symmetric structures of set and geometric configurations, we can phrase the main question of this section as follows.

\begin{quest}
What is the relationship between the automorphism group of a configuration and the symmetries of its possible realizations? Furthermore, ff the automorphism group of configurations acts transitively, do there necessarily exist geometric realizations with non-trivial symmetry? \cite{Grunbaum}
\end{quest}

It is known that any symmetry group of a geometric configuration is either cyclic or dihedral in nature \cite{Grunbaum}. In this section we explore in particular automorphisms with cyclic subgroups and realizations thereof with rotational symmetry. We use the terminology ``having $t$-fold rotational symmetry" interchangeably with ``having symmetry group $\mathbb{Z}t/\mathbb{Z}$" as they are equivalent. 

To begin to investigate this question we necessarily need to be able to identify the automorphism group of a given set configuration. This turns out to be non-trivial so we enlist some tools from graph theory to assist. 
\subsection{The Levi Graph}\hfill\\
\begin{defn}
The \textbf{Levi graph} of a $(n_k)$set-configuration is a bipartite graph of $n$ black and white vertices, corresponding to the points and lines of the configuration, respectively. An edge exists between different colored vertices if and only if the corresponding points and lines are incident to each other. 
\end{defn}

\begin{thm}[\cite{Grunbaum}]
Every set-configuration admits a unique Levi graph. Additionally, there is a 1-to-1 relationship between automorphism of a set-configuration and \emph{graph} automorphism of its Levi graph.
\end{thm}

The Levi graph proves to be an effective tool as its group of graph automorphisms is isomorphic to the automorphism group of the set configuration it represents. Quite a lot is known about graph automorphism and there even exists computational methods for determining an graph automorphism group from an arbitrary graph. A bit of trickiness arises however when we notice that Levi graphs are bipartite and we only concern ourselves with graph automorphism that map same color to same color.

\begin{figure}[H]
\subfloat{\includegraphics[width=0.4\textwidth]{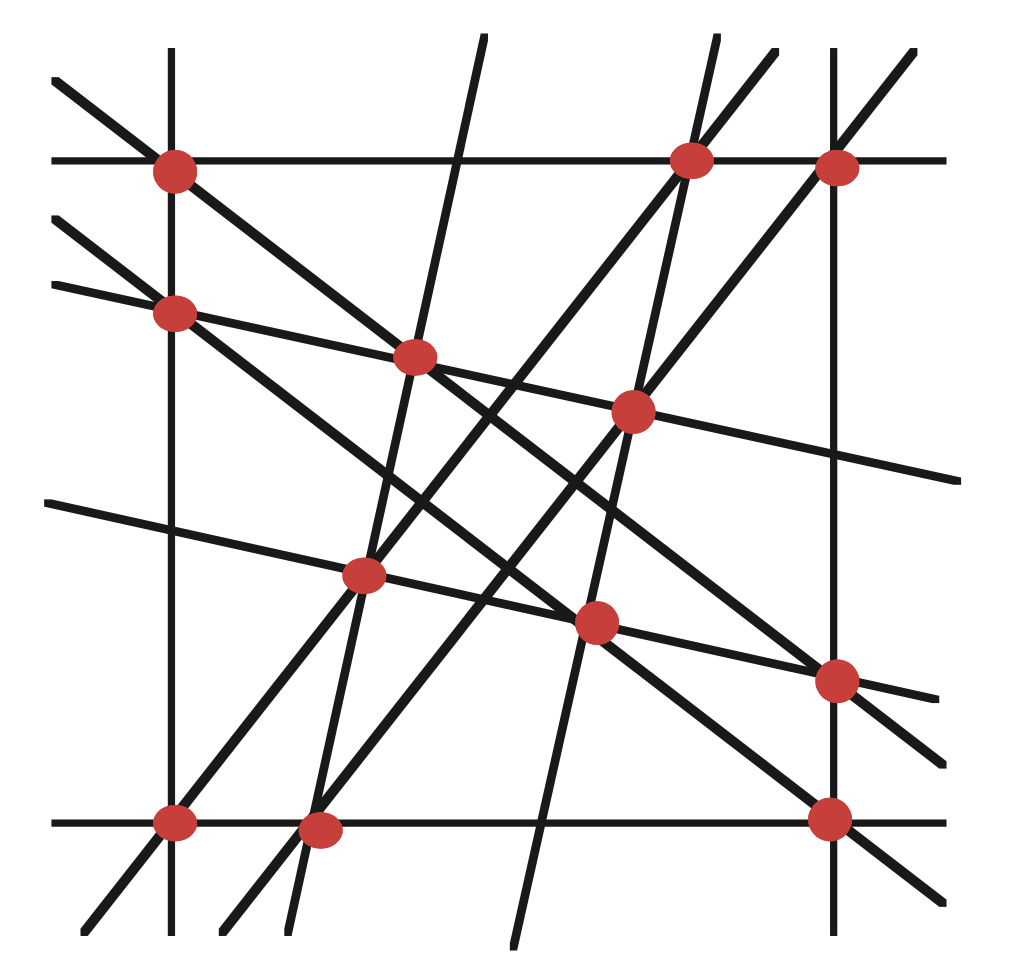}}
\subfloat{\includegraphics[width=0.4\textwidth]{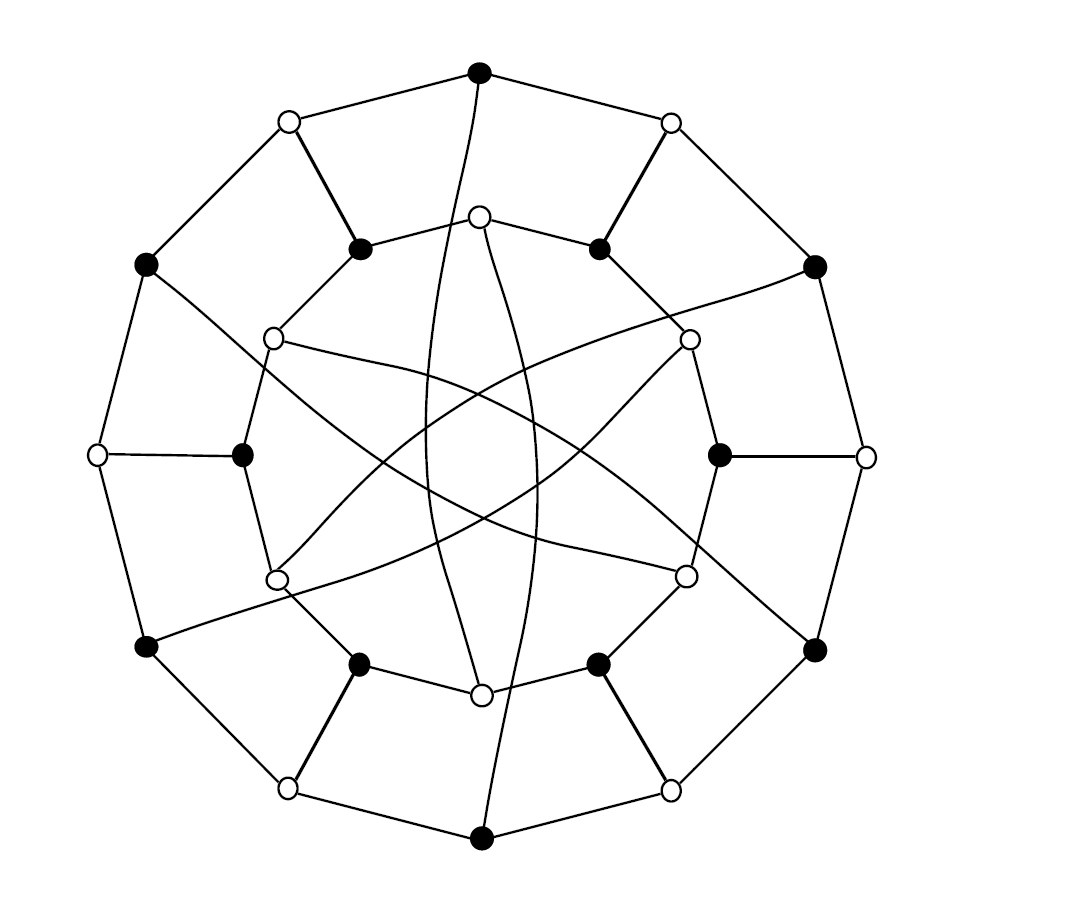}}
{\caption{A $(12_3)$ configuration and its Levi Graph.}\cite{Grunbaum}}
\end{figure}

\section{Examples of Automorphisms and Symmetries}\hfill\\
Recall that our interest is in the relationship between the automorphism group of a set-configuration and the symmetries of its possible realizations. In this section we choose several examples of set-configurations with known geometric realizations and determine their automorphism group. 
\subsection{The $(9_3)_2$ Configuration}\hfill\\

We begin with the Levi graph and set-configuration for the second of three non-isomorphic $(9_3)$ configurations.
\begin{figure}[H]
\subfloat{\includegraphics[width=0.5\textwidth]{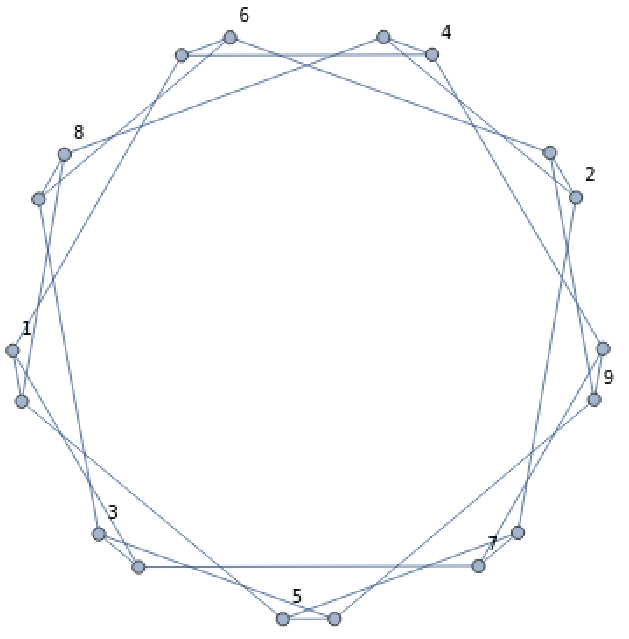}}
\subfloat{\resizebox{.5\linewidth}{!}{
    \begin{tabular}{|c c c c c c c c c |}
    \hline
    1 & 1 & 1 & 2 & 2 & 2 & 3 & 3 & 4\\ 
    \hline
    3 & 4 & 5 & 4 & 5 & 6 & 5 & 6 & 7\\
    \hline
    7 & 6 & 8 & 8 & 7 & 9 & 9 & 8 & 9\\
    \hline
    \end{tabular}}}
\caption{The $(9_3)_2$ set-configuration and its Levi graph (lines unlabeled).}
\label{fig:932}
\end{figure}

Using Figure \ref{fig:932}, we demonstrate how we can detect a few subgroups of $\text{Aut}((9_3)_2)$ by inspection.

\begin{thm}
$\mathbb{Z}/9\mathbb{Z} \leq \text{Aut}((9_3)_2))$.
\end{thm}

\begin{proof}
The permutation that maps point $i$ to points $i + 2 (\text{mod} 9)$ is an automorphism. Furthermore this automorphism is cyclic and thus generates a subgroup of group $\text{Aut}((9_3)_2))$ isomorphic to $\mathbb{Z}/9\mathbb{Z}$. 
\end{proof}

In this case, $\mathbb{Z}/9\mathbb{Z}$ acts transitively on the set of edges. We might suspect that there exists a geometric realization of $(9_3)_2$ with some form of rotational symmetry. One such realization is shown in Figure \ref{fig:932}

\begin{figure}[H]
\includegraphics[width=0.5\linewidth]{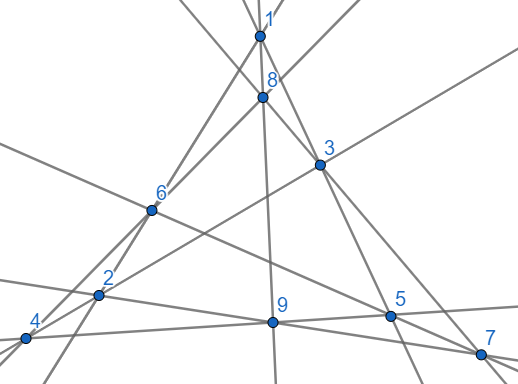}
\caption{A realization of the $(9_3)_2$ configuration.}
\label{fig:932}
\end{figure}

Before moving on, we note that, as it turns out, any ($n_3$) configuration $C$ with $\mathbb{Z}/n\mathbb{Z} \leq \text{Aut}(C)$ is a special case of configurations called cyclic configurations which we discuss in the latter half of this paper. Additionally, it can be shown that, when $n$ is even, the geometric realizability of such configurations as symmetric in the euclidean plane is well studied.

\subsection{Pappus}\hfill\\

One of the earliest discovered geometrically realizable 3-configurations is the Pappus configuration (commonly denoted as $(9_3)_ 1$). Here, we investigate the structure of its automorphism group and its implications on certain geometric realizations.

\begin{figure}[H]
\includegraphics[width=0.9\linewidth]{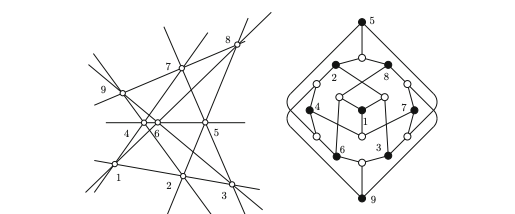}
\caption{The Pappus configuration $(9_3)_1$.  On the left is the geometric configuration, and on the right is its corresponding Levi graph.\cite{graphs}}
\end{figure}

\begin{thm}
$(\mathbb{Z}/3\mathbb{Z})^2 \leq \text{Aut}((9_3)_1)$.
\end{thm}
\begin{proof}
It is easy to see that the permutation of vertices $(123)(456)(789)$ is an automorphism of the Levi graph. Additionally the permutation $(148)(258)(369)$ is an automorphism, and commutes with the first. These two permutations generate a subgroup of $Aut((9_3)_1)$ of order 9 isomorphic to $(\mathbb{Z}/3\mathbb{Z})^2$ \cite{graphs}.
For any given point in the configuration, that point is not fixed by any of these permutations other than the identity. This implies every element has a trivial stabilizer. Since $|(\mathbb{Z}/3\mathbb{Z})^2|$ is equal to the number of vertices we can apply the Orbit-Stabilizer theorem and see that this subgroup acts transitively on the points of the configuration.

\end{proof}

One might expect the existence of a geometric realization of the Pappus configuration with 3-fold rotational symmetry. This is observed in Figure \ref{fig:pappus}.

\begin{figure}[H]
\includegraphics[width=0.6\linewidth]{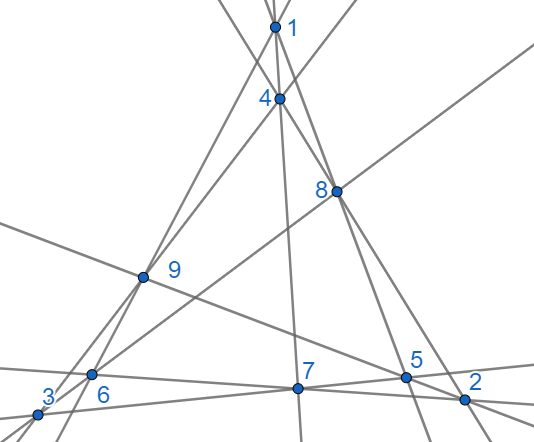}
\caption{A realization of the Pappus Configuration with 3-fold rotational symmetry.}
\label{fig:pappus}
\end{figure}

\subsection{Desargues}\hfill\\

Another early configuration is the Desargues configurations also known as $(10_3)_ 1$ and we endeavor to explore its automorphic structure in the same way. In Figure \ref{fig:des_levi} we see its Levi graph.

\begin{figure}[H]
    \centering
    \includegraphics{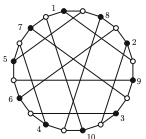}
    \caption{The Levi graph for the $(10_3)_ 1$ configuration \cite{graphs}.}
    \label{fig:des_levi}
\end{figure}

\begin{thm}
$\mathbb{Z}/5\mathbb{Z} \leq \text{Aut}((10_3)_1)$
\end{thm}

\begin{proof}
We can see that the permutation $(12345)(67890)$ is an automorphism and generates a cyclic subgroup of order 5 isomorphic to $\mathbb{Z}/5\mathbb{Z}$. 
\end{proof}

This action has two orbits and is not transitive. Unlike our previous examples, there is no known rotationally symmetric realization of $(10_3)_ 1$ in the standard Euclidean plane. In Figure \ref{fig:des_levi} we see a geometric realization (with no apparent symmetry).

\begin{figure}[H]
\centering
    \includegraphics[width=0.5\linewidth]{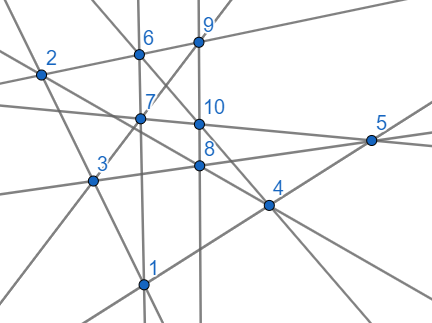}
    \caption{Geometric Realization of the Desargues configuration.}
    \label{fig:des_levi}
\end{figure}

\subsection{The $(10_3)_{10}$ Configuration}\hfill\\

In Figure \ref{fig:iso1} we displayed the set-configurations of $(10_3)_{10}$ and two of its geometric realizations with 5-fold and 2-fold rotational symmetry, respectively. In Figure \ref{fig:103levi}, we give its Levi graph.

\begin{thm}
$\mathbb{Z}/10\mathbb{Z} \leq \text{Aut}((10_3)_{10})$
\end{thm}

\begin{proof}
The permutation that maps vertex $i$ to vertex $i + 1\ (\text{mod } 10)$ is an automorphism when.  This permutation generates a subgroup of order 10 isomorphic to $\mathbb{Z}/10\mathbb{Z}$ that acts transitively on the points (and lines) of the configuration.
\end{proof}

 We know that this group is abelian and thus by the Fundamental Theorem of Finite Abelian Groups it must be isomorphic to $\mathbb{Z}/5\mathbb{Z} \times \mathbb{Z}/2\mathbb{Z}$. This serves as another positive example where a subgroup of an automorphism group is manifested in the symmetries of its realizations.

\begin{figure}[H]
\includegraphics[width=0.5\textwidth]{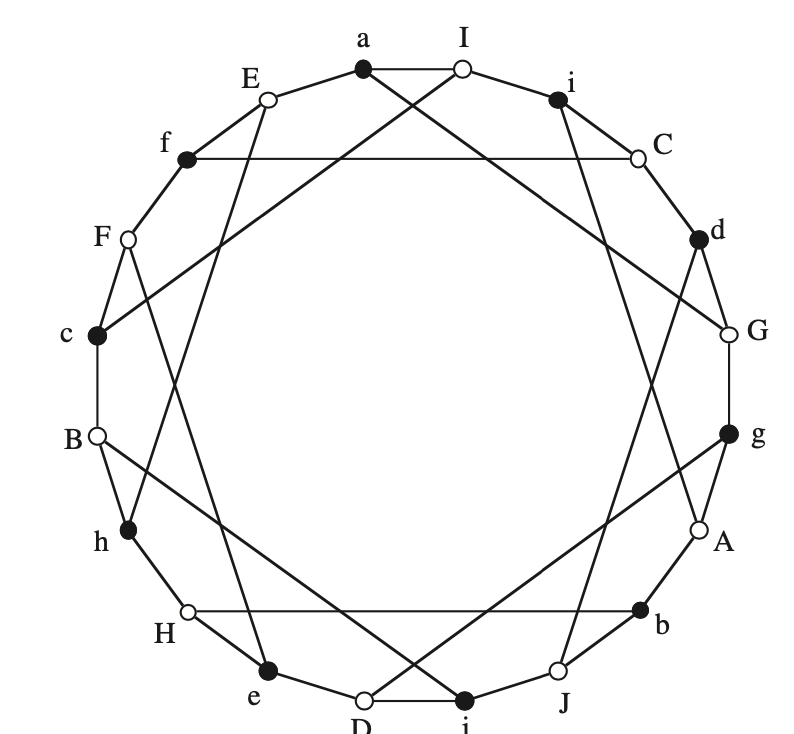}
\caption{Levi graph of the $(10_3)_{10}$ configuration.}\cite{Grunbaum}
\label{fig:103levi}
\end{figure}

\section{Generalized Cyclic Configurations}
\subsection{Proofs of Existences}\label{sec:exist}\hfill\\

While a complete enummeration of the various categories of configurations for $n$ small can be given, it is perhaps more foundational and constructive to highlight the ways to show that these configurations exist.  The theorems below do so, and were known as early as 1882.  

\begin{thm} 
Combinatorial configurations $(n_3)$ exist if and only if $n \geq 7$.
\label{thm:cconf}
\end{thm}

\begin{thm}
Topological configurations $(n_3)$ exist if and only if $n \geq 9$.
\label{thm:tconf}
\end{thm}

\begin{thm}
Geometric configurations $(n_3)$ exist if and only if $n \geq 9$.
\label{thm:gconf}
\end{thm}

Over the years, the theorems above have been proven in several different manners.  Here, we summarize the proofs given by Grunbaum
for Theorems \ref{thm:cconf} and \ref{thm:gconf}, as these will be useful for the remainder of this paper \cite{Grunbaum}.  To do so, we first introduce a necessary definition.\\

\begin{defn}
A \textbf{generalized cyclic configuration} is one that is illustrated using Table \ref{tab:gencycl}.  This table is denoted as $\mathscr{C}(n,a,b)$, $1 \leq a < b < n$

\begin{table}[h]
\centering
    \begin{tabular}{|c c c c c c c c c |}
    \hline
    $1$ & $2$ & $3$ & $4$ & $\hdots$ & $n-3$ & $n-2$ & $n-1$ & $n$\\ 
    \hline
    $1+a$ & $2+a$ & $3+a$ & $4+a$ & $\hdots$ & $n-3 + a$ & $n-2 + a$ & $n-1+a$ & $n+a$\\
    \hline
    $1+b$ & $2+b$ & $3+b$ & $4+b$ & $\hdots$ & $n-3 + b$ & $n-2 + b$ & $n-1+b$ & $n+b$\\
    \hline
    \end{tabular}
    \caption{}
\label{tab:gencycl}
\end{table}
\end{defn}

For $n \geq 7$, $\mathscr{C}(n,1,3)$ is a valid combinatorial configuration, as no two columns share the two or more of the same points, and each of the $n$ lines are incident with 3 points (and vice-versa).  This completes the proof of Theorem $\ref{thm:cconf}$.\\

To show Theorem $\ref{thm:gconf}$, we again consider $C = \mathscr{C}(n,1,3)$.  \\

We note that $C$ contains a path $2,3,4,...,n-2$, where each point is incident with the point immediately to the left and the right of it.  Moreover, for $1 \leq i < n-5$, the points at index $i, i+1,$ and $i + 3$ are those points which are collinear.  With these in mind, the points of this path can be placed as follows:
\begin{enumerate}
    \item Place point 2 at the origin, and point 4 somewhere to the left of point 2.
    \item Draw a line through 2 with small positive slope, and place point 3 on this line, such that it is close to 2, and the line $3,4$ has small negative slope.  Also place on this line point 5, such that the line $4,5$ has positive slope.
    \item On line $3,4$, place point 6 such that its $x$-coordinate is larger than the $x$-coordinate of point 5.
    \item On line $4,5$, place point 7 such that its $x$-coordinate is larger than the $x$-coordinate of point 6. 
    \item Proceed in a similar fashion until point $n-2$ is placed on the line $n-4, n-5$, such that its $x$-coordinate is larger than the $x$-coordinate of point $n-3$.  
\end{enumerate}
While the complete proof can be found in the aforementioned reference, for our purposes, it is enough to note that remaining three points $n-1$, $n$, and $1$ can always be placed to fulfil the necessary remaining incidences, which completes the proof of Theorem $\ref{thm:gconf}$.  
\newpage
\begin{exmp}
Below is the configuration $\mathscr{C}(12,1,3)$, generated using the algorithm above.
\begin{figure}[H]
    \centering
    \includegraphics[width=0.5\textwidth]{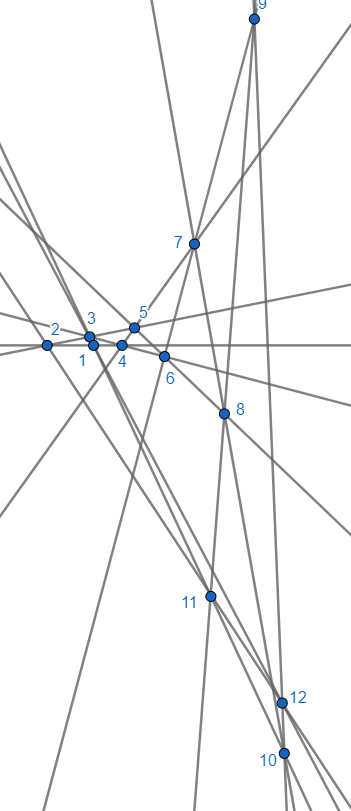}
    \caption{}
\end{figure}
\end{exmp}

\subsection{Generalizations of Section \ref{sec:exist}}\hfill\\

For this section, our motivating question is this:
\begin{quest}
For which values of $a$ and $b$ can the generalized cyclic table $\mathscr{C}(n,a,b)$ be used to prove Theorems \ref{thm:cconf} and \ref{thm:gconf}?
\end{quest}

The values of $a,b$ for which the combinatorial configuration $\mathscr{C}(n,a,b)$ exists have been known since the late 1800s (see,for instance, \cite{Brunel}).  However, we provide our own version of the proof below:

\begin{thm}\label{thm:cconf_exist}
For a fixed $n$ and the generalized cyclic table $C = \mathscr{C}(n,a,b), a < b$, \\if $b \not \in \{n - a, \frac{n+a}{2}, \frac{n}{2}+a, \frac{n}{2}, 2a\}$ and $a \not = \frac{n}{2}$, then $C$ is a combinatorial configuration. 
\end{thm}

\begin{proof}
Let $A, B$ be two rows of $C$, such that $A\not = B$.  Let $A = \{j, j + a, j+b\}$ and $B = \{k, k+a,k+b\}$.  By definition, in order for $C$ to be a combinatorial configuration, $|A \cap B|$ must be one or less.  Therefore, we consider the case when $C$ is not a combinatorial configuration (i.e. we consider the case when $|A \cap B| \geq 2$).  \\

Without loss of generality, as $|A\cap B| \geq 2$, there are three distinct cases we need to consider, which are outlined below.\\

\begin{itemize}
    \item If $j \equiv k + a\mods$, then either $j + a \equiv k \mods$ or $j + a \equiv k + b\mods$.  Rearranging, we get $j - k \equiv a \mods$ and $(j - k \equiv -a\mods$ or $ j - k \equiv b - a\mods)$.  Simplifying slightly, we get $a \equiv -a\mods$ or $a \equiv b-a\mods$.  Finally, this gives us $a = \frac{n}{2}$ or $b = 2a$.\\
    
    \item If $j+a \equiv k + b\mods$, then either $j + b \equiv k+a\mods$ or $j + b \equiv k\mods$.  Rearranging, we get $j - k \equiv b-a\mods$ and $(j - k \equiv a-b\mods$ or $j - k \equiv -b\mods)$.  Simplifying slightly, we get $b-a \equiv a-b\mods$ or $b-a \equiv -b\mods$.  Finally, this gives us $b = \frac{n}{2}$ or $b = \frac{a+n}{2}$. \\
    \item If $j+b \equiv k\mods$, then either $j \equiv k+b\mods$ or $j \equiv k+a\mods$.  Rearranging, we get $j - k \equiv -b\mods$ and $(j - k \equiv b\mods$ or $j - k \equiv a\mods$.  Simplifying slightly, we get $-b \equiv b\mods$ or $-b \equiv a\mods$.  Finally, this gives us $b = \frac{n}{2}$ or $b = n - a$.\\
\end{itemize}
And so, when $|A \cap B| \geq 2$, $b$ must be an element of $\{n - a, \frac{n+a}{2}, \frac{n}{2}+a, \frac{n}{2}, 2a\}$ or $ a = \frac{n}{2}$.\\

Therefore, when $b \not \in \{n - a, \frac{n+a}{2}, \frac{n}{2}+a, \frac{n}{2}, 2a\}$ and $a \not = \frac{n}{2}$, $|A \cap B| < 2$, indicating the table represents a combinatorial configuration, as claimed.  
\end{proof}

\begin{exmp}
We can confirm these results graphically in Figure \ref{fig:invalid}, by generating all invalid generalized cyclic tables, $8 < n < 50$.  For a fixed $n$, the points plotted are the integer values $a$ and $b$, $a < b$, which satisfy the equations $n = 2b - 2a, n = 2a, n = 2b, n = a + b, n = 2b - a, \text{ or } 2a = b$.  These results are discussed more in-depth in Section \ref{sec:centroid}.
    \begin{figure}[H]   
    \centering
     \includegraphics[width=.8\linewidth]{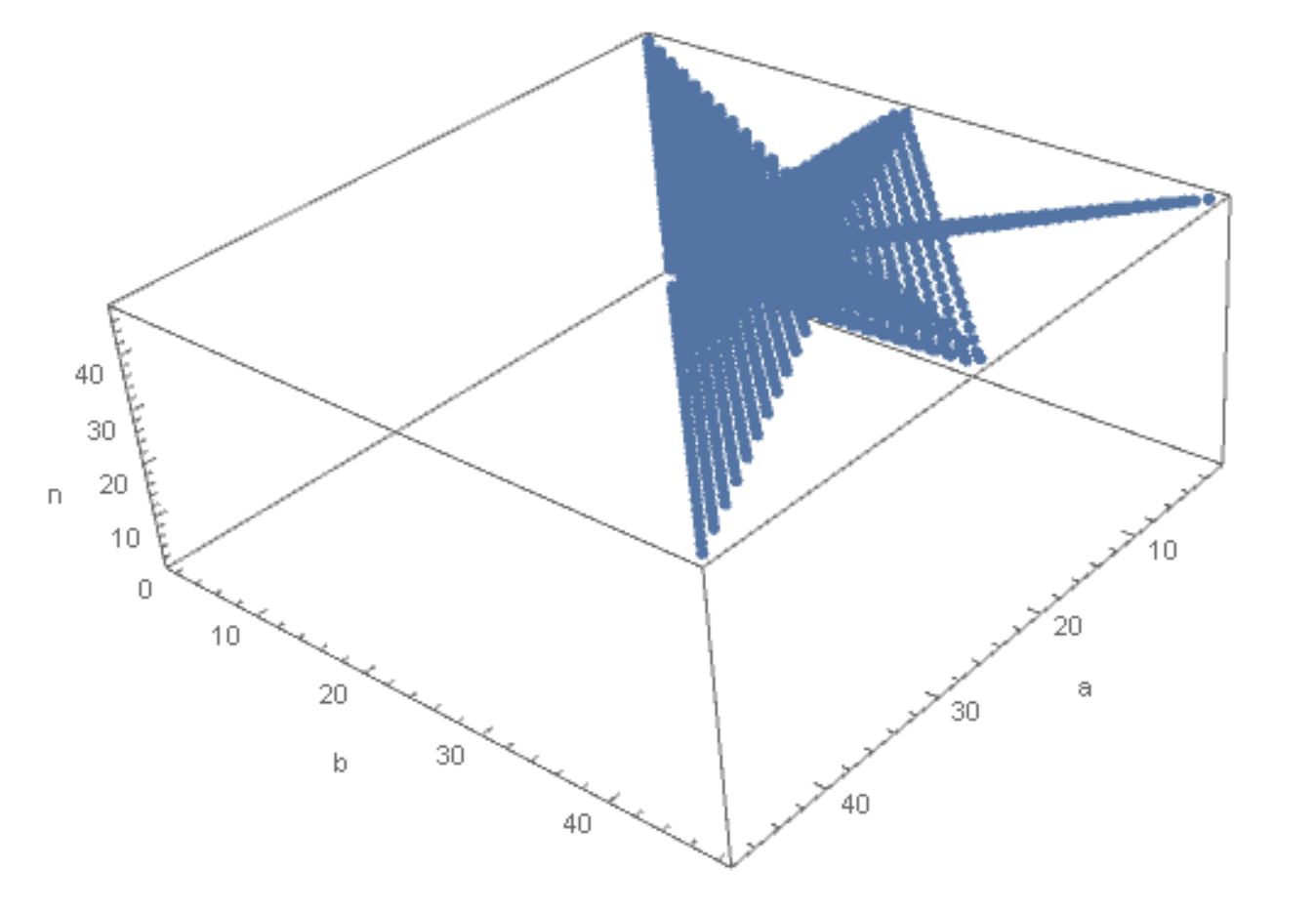}
     \caption{}
     \label{fig:invalid}
    \end{figure}

\end{exmp}

Regarding the results of Theorem \ref{thm:cconf_exist}, there is a necessary caveat to be made.  While any value $a,b$ that satisfies the criteria will generated a valid combinatorial configuration, this does not mean that any two pairs will immediately yield two non-isomorphic combinatorial configurations.  Indeed, isomorphism classes of generalized cyclic configurations can be approached as follows:
\begin{thm}[\cite{chiral}]\label{thm:cyciso}
Any two generalized cyclic configurations, $C_1$ and $C_2$, are isomorphic if and only if all values of $C_2$ are the same constant multiple of $C_1$ (mod $n$).
\end{thm}

We explore the usefulness of such a theorem in the example below.
\begin{exmp}
It can be shown that, up-to-isomorphism, there exists only three $(9_3)$ combinatorial configurations (see, for instance, \cite{Grunbaum})  These three configuration tables are shown below.
\begin{table}[H]
\centering
    \begin{tabular}{|c c c c c c c c c|}
    \hline
    1 & 1 & 1 & 2 & 2 & 2 & 3 & 3 & 3\\ 
    \hline
    4 & 5 & 6 & 4 & 5 & 6 & 4 & 5 & 6\\
    \hline
    7 & 8 & 9 & 8 & 9 & 7 & 9 & 7 & 8\\
    \hline
    \end{tabular}
    \caption{The configuration $(9_3)_1$.}
\label{tab:931}

\centering
    \begin{tabular}{|c c c c c c c c c|}
    \hline
    1 & 1 & 1 & 2 & 2 & 2 & 3 & 3 & 4\\ 
    \hline
    3 & 4 & 5 & 4 & 5 & 6 & 5 & 6 & 7\\
    \hline
    7 & 6 & 8 & 8 & 7 & 9 & 9 & 8 & 9\\
    \hline
    \end{tabular}
    \caption{The configuration $(9_3)_2$.}
\label{tab:932}

\centering
    \begin{tabular}{|c c c c c c c c c|}
    \hline
    1 & 1 & 1 & 2 & 2 & 2 & 3 & 3 & 3\\ 
    \hline
    4 & 5 & 8 & 4 & 5 & 7 & 4 & 6 & 7\\
    \hline
    7 & 6 & 9 & 6 & 8 & 9 & 5 & 9 & 8\\
    \hline
    \end{tabular}
    \caption{The configuration $(9_3)_3$.}
\label{tab:933}
\end{table}
It can be easily shown that, by Theorem \ref{thm:cyciso}, $\mathscr{C}(9,1,3)$, $\mathscr{C}(9,2,6)$, $\mathscr{C}(9,3,4)$, $\mathscr{C}(9,5,6)$, $\mathscr{C}(9,3,7)$, and $\mathscr{C}(9,6,8)$ are isomorphic to each other.  Furthermore, by observation, we see that $\mathscr{C}(9,2,6)$ is isomorphic to the configuration $(9_3)_2$.
\begin{table}[H]
    \centering
    \begin{tabular}{|c c c c c c c c c|}
    \hline
    2 & 4 & 6 & 8 & 1 & 3 & 5 & 7 & 9\\ 
    \hline
    4 & 6 & 8 & 1 & 3 & 5 & 7 & 9 & 2\\
    \hline
    8 & 1 & 3 & 5 & 7 & 9 & 2 & 4 & 6\\
    \hline
    \end{tabular}
    \caption{The configuration $\mathscr{C}(9,2,6)$, which is isomorphic to both $\mathscr{C}(9,1,3)$ and $(9_3)_2$. }
\end{table}
\vspace{-5mm}
Indeed, all valid $\mathscr{C}(9,a,b)$ configurations can be shown to be isomorphic to one of these three configurations.  \\
\end{exmp}

Theorem \ref{thm:cyciso} is a powerful tool that can be used to generate isomorphic cyclic configurations.  However, for our purposes, the following corollary is far more useful:
\begin{cor}
Given two generalized cyclic configurations, $C_1 = \mathscr{C}(n,a_1,b_1)$ and $C_2 = \mathscr{C}(n,a_2,b_2)$, if there exists $ 1<z<n$ such that $b_2 \equiv zb_1$ (mod $n$) and $a_2 \equiv za_1$ (mod $n$), then $C_1$ and $C_2$ are isomorphic.  
\end{cor}
\begin{proof}
By Theorem \ref{thm:cyciso}, if $C_1$ and $C_2$ are isomorphic, then for a column $A = \{j, j+a_1, j+b_1\} \subset C_1$, there exists column $B \subset C_2$ such that $B = \{zj, zj+za_1, zj+zb_1$, for $1 < z < n$. \\
Notice that $a_2 \equiv zj+za_1 - zj = za_1$ (mod $n$) and $b_2 \equiv zj+zb_1 - zj = zb_1$ (mod $n$), as claimed.\\

Given the convention that $a < b$, it may be necessary to swap the rows 2 and 3.  However, doing so will still yield an isomorphic configuration, and so we are done.
\end{proof}

Now, we turn our attention to the values of $a,b$ such that the generalized cyclic configuration can be geometrically realized.  Specifically, we focus on the values of $a,b$ such that this geometric realization can be done using the same proof by construction as that presented for Theorem \ref{thm:gconf}.  However, this, as it turns out, is not difficult to show, at least partially.  Indeed, for a fixed $n$, any such pair that generates an isomorphic configuration to $\mathscr{C}(n,1,3)$ will be geometrically realizable in this manner.  \\

For values $a,b$ that do not satisfy this criteria, it does not necessarily mean that this algorithm will not work, although we suspect it does not.  However, it has been shown that there are other ways to geometrically realize these types of configurations.  But first, some necessary definitions.
\begin{defn}
    A configuration $C$ is called \textbf{astral} if the number of orbits of points is equal to the number of orbits of lines
\end{defn}
\begin{defn}
    A configuration $C$ is called \textbf{chiral} if $C$ is astral, and is with only a cyclic group of symmetries.
\end{defn}

We conclude this section by highlighting a result by Berman et al:\\

\begin{thm}[\cite{chiral}]
For $n = 2m, m \in \mathbb{N}$, $\mathscr{C}(n,a,b)$ can be geometrically realized as a chiral astral configuration (with a few exceptions) if $a$ is even, $a < m$, $b$ is odd, and $0 < b < a$ or $m < b < a+m$. 
\end{thm}




\section{Further Exploration of Generalized Cyclic Configurations}
In this section, we explore some results found as a result of the work done in the above sections, but that are not necessary useful in those endeavours.  

\subsection{Centroids and Generalized Cyclic Combinatorial Configurations\label{sec:centroid}}\hfill\\
\begin{exmp}
In the figures below, we plot the values of $(a,b)$ such that, for a fixed $n$, $\mathscr{C}(n,a,b)$ is not a valid combinatorial configuration.
\begin{figure}[H]  
\captionsetup[subfloat]{labelformat=empty}
\subfloat[$\mathscr{C}(40,a,b)$.]{\includegraphics[width=0.4\textwidth]{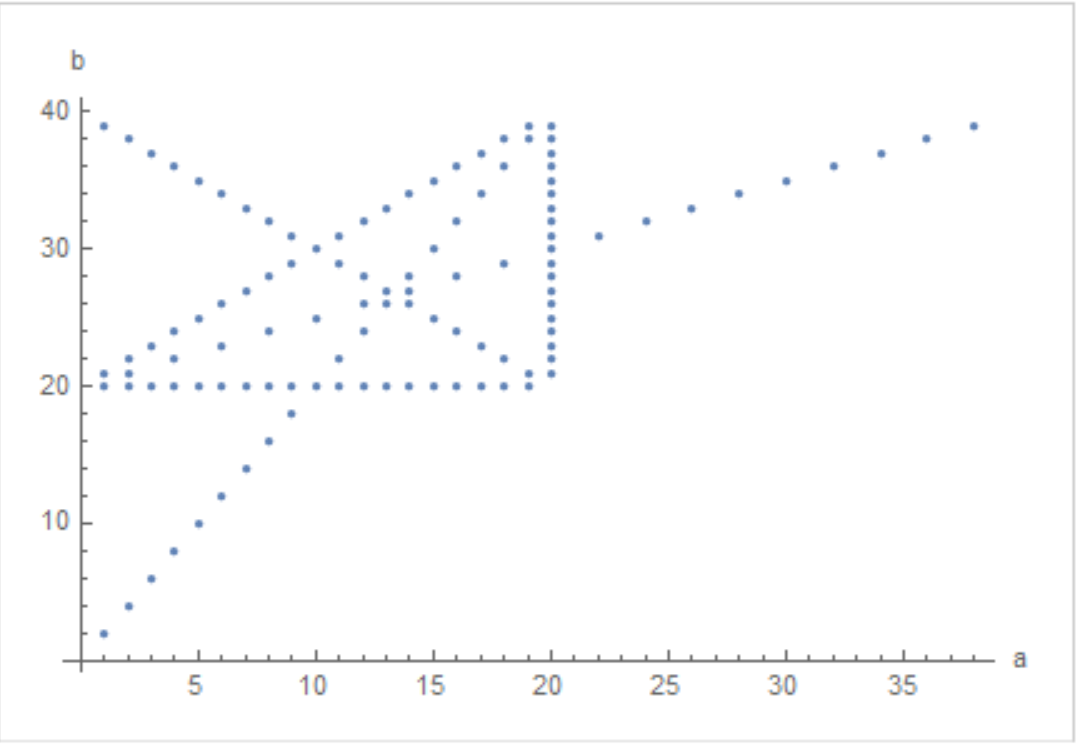}}
\subfloat[$\mathscr{C}(43,a,b)$.]{\includegraphics[width=0.4\textwidth]{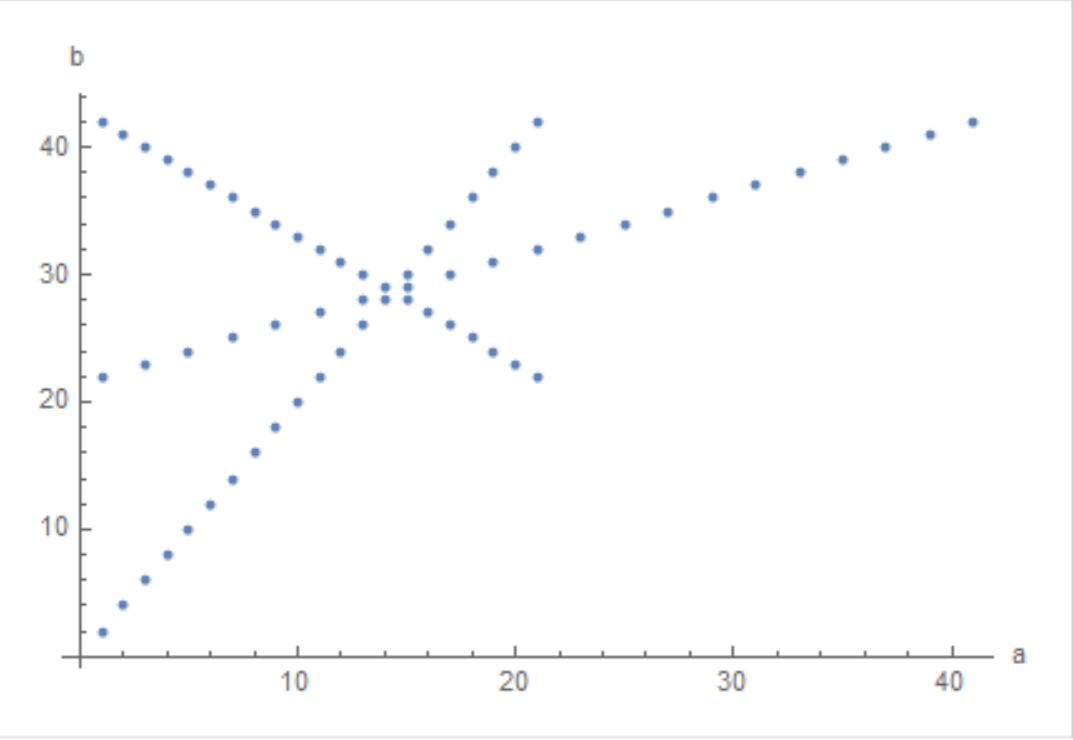}}
\end{figure}
\vspace{-5mm}
\begin{figure}[H]    
\captionsetup[subfloat]{labelformat=empty}
\subfloat[$\mathscr{C}(30,a,b)$.]{\includegraphics[width=0.45\textwidth]{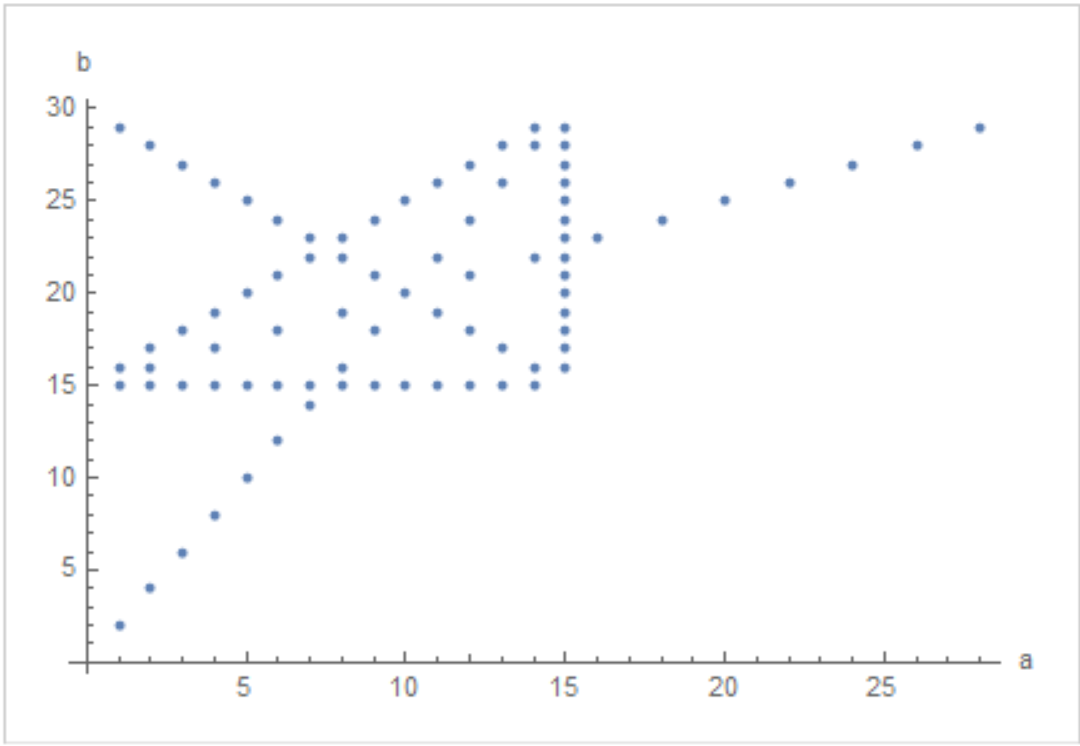}}
\end{figure}
From the examples above, it appears that the centroid of the triangle bounded by three lines is exactly the intersection of the remaining three lines.  Moreover, while not shown, the centroid $(a',b')$ yields the only configuration (for that specific $n$) such that the same three points appear on two distinct lines.  Put differently, the configuration $\mathscr{C}(n,a',b')$ is the only such configuration for this $n$ that contains two columns $A$ and $B$ whose intersection is size three.  
\end{exmp}

We attempt to explain this observation in the remainder of this section.\\
But first, we present a definition used several times throughout this explanation.
\begin{defn}
    Let $A$ and $B$ be columns of a configuration $C = \mathscr{C}$, and let $A[i]$ represent the $i$-th element of the line $A$.  We say that $A$ and $B$ are equal (or $A = B$) if, for all 1 $\leq j \leq 3, A[j] = B[j]$.
\end{defn}

Now, consider this restatement of Theorem \ref{thm:cconf_exist}:  
For a fixed $n$, if $(a,b)$ is a point on the line $n = 2b - 2a, n = 2a, n = 2b, n = a + b, n = 2b - a, \text{ or } 2a = b$, then $C' = \mathscr{C}(n,a,b)$ is not a configuration.

\vs
Using this as a starting point, we proceed with the following theorem:

\begin{thm}\label{thm:centroid}
For $n$ even, if $C' = \mathscr{C}(n,a,b)$, $b > a$ is not a combinatorial configuration, then the specific $(a',b')$ pair that yields $|A\cap B| = 3$, for $A,B \subset C', A \not = B$ is the centroid of the triangle bounded by the lines $n = 2b - 2a, n = 2a, n = 2b$. \\  Moreover, this point satisfies the equations $n = a + b, n = 2b - a, 2a = b$.
\end{thm}
\begin{proof}

By Lemma \ref{lem:intersect}, $|A \cap B| = 3$ implies that $a' = \frac{n}{3}$ and $b' = \frac{2n}{3}$. \\ For a fixed $n$, the lines $n = 2b - 2a, n = 2a, n = 2b$ form a right triangle.  It follows that the coordinates of the three vertices of this triangle are $(0,\frac{n}{2}), (\frac{n}{2},\frac{n}{2})$ and $(\frac{n}{2},n)$.  We can then calculate the centroid to be $(\frac{n}{3}, \frac{2n}{3})$.  The coordinates are exactly the values of $a'$ and $b'$.
Additionally, it is obvious that the point $(a',b')$ is exactly the intersection of the lines $n = a + b, n = 2b - a, 2a = b$.
\end{proof}

\begin{lem}\label{lem:intersect}
Let $A,B$ both be columns of $\mathscr{C}(n,a,b)$, $b > a$, such that $A \not = B$. \\ Then, for all $A$, there exists $B$ such that $|A\cap B| = 3$ if and only if $3|n \text{ and } a = \frac{n}{3}$ and $b = \frac{2n}{3}$
\end{lem}

\begin{proof}
We begin by proving the forward implication.
Let $A = [1,1+a,1+b]$.  \\Choose a $B = [k,k+a,k+b]$, for some $1  < k \leq n$, such that $|A \cap B | = 3.$
As $1<k$ and $a<b$, the intersection of $A$ and $B$ is size 3 only if $[1 \equiv k+a \land 1+a \equiv k+b \land 1 + b \equiv k]$ (all mod $n$).
Simplifying, we obtain $[a = n - b \land a = b - a]$.  Since $a,b,n \in \mathbb{N}$, it immediately follows that $3 | n$ , and $a = \frac{n}{3}, b = \frac{2n}{3}$, as claimed.
\\To prove the reverse implication, suppose that $3 | n$ and $a = \frac{n}{3}$, $b = \frac{2n}{3}$.  Let $C \in  \mathscr{C}(n,a,b)$, i.e. $C = [q, q + a, q + b] = [q, q + \frac{n}{3}, q + \frac{2n}{3}]$.  Construct another line $D$ which contains all points of $C$ shifted by $\frac{n}{3}$, i.e. $D = [q + \frac{n}{3}, q + \frac{2n}{3}, q + n]$.  Recognizing all operations are mod $n$, and letting $r = q + \frac{n}{3}$, we obtain $D = [r, r + a, r + b]$, which is an element of $\mathscr{C}(n,a,b)$.  Since $C \not = D$, but $|C \cap D| = 3$, we are done.
\end{proof}

\begin{cor}
It follows from Theorem \ref{thm:centroid} (and the subsequent lemma) that, for a specific $n = 2m$, $m \in \mathbb{N}$, every point $(a,b)$ (excluding the centroid) gives rise to a generalized cyclic table $C' = \mathscr{C}(n,a,b)$ such that, for $A, B \subset C'$, $|A\cap B| = 2$. \qedsymbol
\end{cor}

Before moving on, we summarize the results of this section.  For a fixed $n$:
\begin{itemize}
    \item From Lemma \ref{lem:intersect}, if $3|n$, the intersection $(a',b')$ of the lines $n = a + b, n = 2b - a, 2a = b$ determines the values $a = a'$ and $b = b'$ which cause $\mathscr{C}(n,a,b)$ to have at least two columns with size three intersection.
    \item From Theorem \ref{thm:centroid}, if $6|n$, the centroid $(a',b')$ of the triangle bounded by the lines $n = 2b - 2a, n = 2a, n = 2b$ determines the values $a = a'$ and $b = b'$ which causes $\mathscr{C}(n,a,b)$ to have at least two columns with size three intersection.
\end{itemize}

\subsection{Generation of Geometric Configurations}\hfill\\

Unsurprisingly, due to the nicheness of this topic, there, from what we could tell, does not exist any specific software tools that would useful when generating geometric configurations. And so, we attempted to make one.  Using Python and relying heavily on the \textit{pygame} package, we designed an animation tool\footnote{Found here: https://github.com/nipayne/Configurations} that would allow for the geometric generation of any cyclic configuration $\mathscr{C}(n,1,3)$, $n \geq 9$.  This script was designed to serve two primary purposes:
\begin{itemize}
    \item First, this was intended to be an education tool.  Grumbaum's algorithm is sometimes difficult to parse, and it very difficult to implement when $n$ is large.
    \item Second, this was intended to lay the foundations and allow for construction techniques using other algorithms.  
\end{itemize}

As an added bonus, the way this tool was designed, configurations can be drawn by hand.  While this can be done in pre-existing software packages such as GeoGebra, we found that these were not as intuitive.  

\begin{figure}[H]
\subfloat[The start of a free-drawn configuration.]{\includegraphics[width=0.5\textwidth]{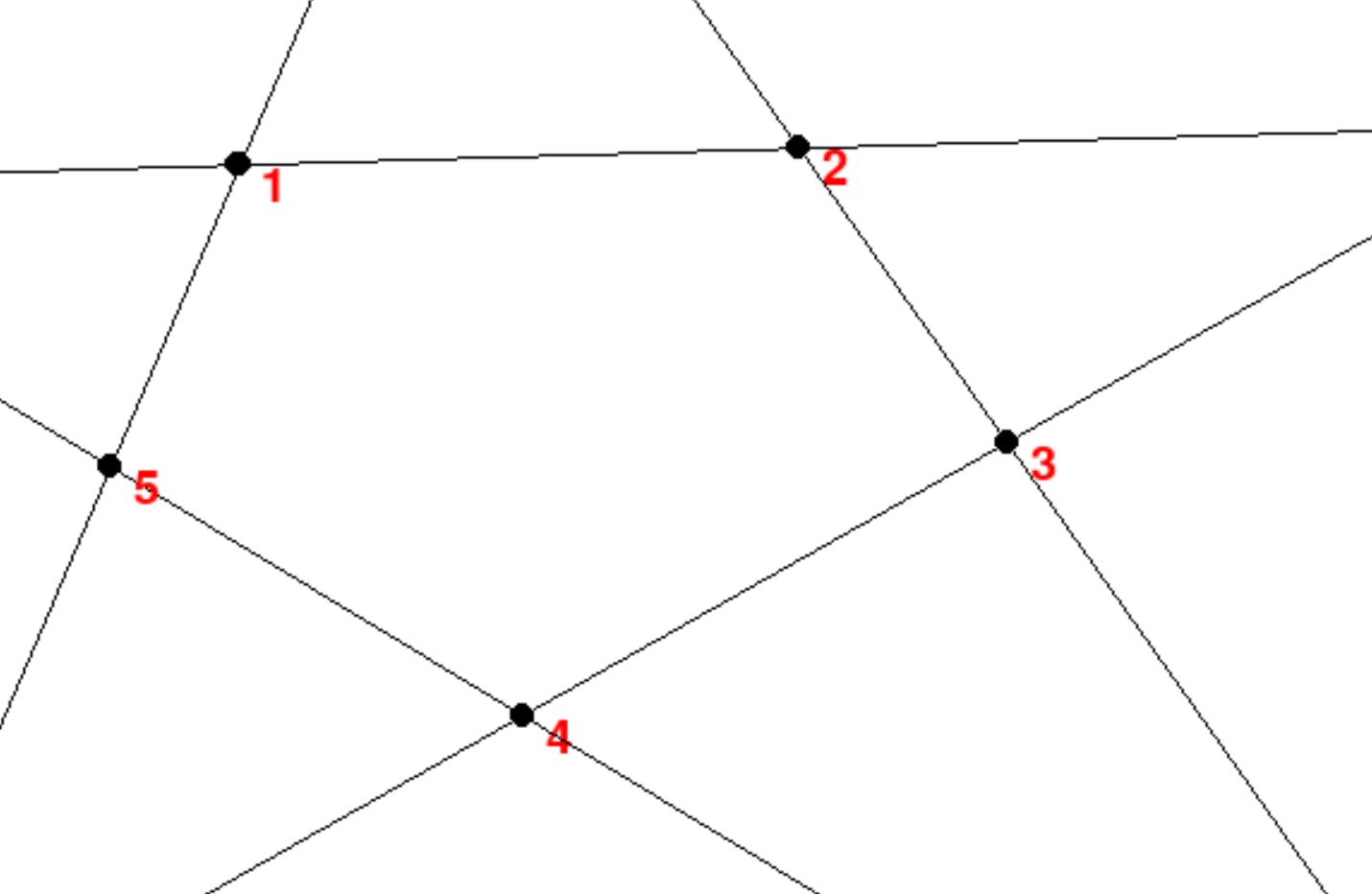}}
\subfloat[The geometric realization of $\mathscr{C}(9,1,3)$.]{\includegraphics[width=0.5\textwidth]{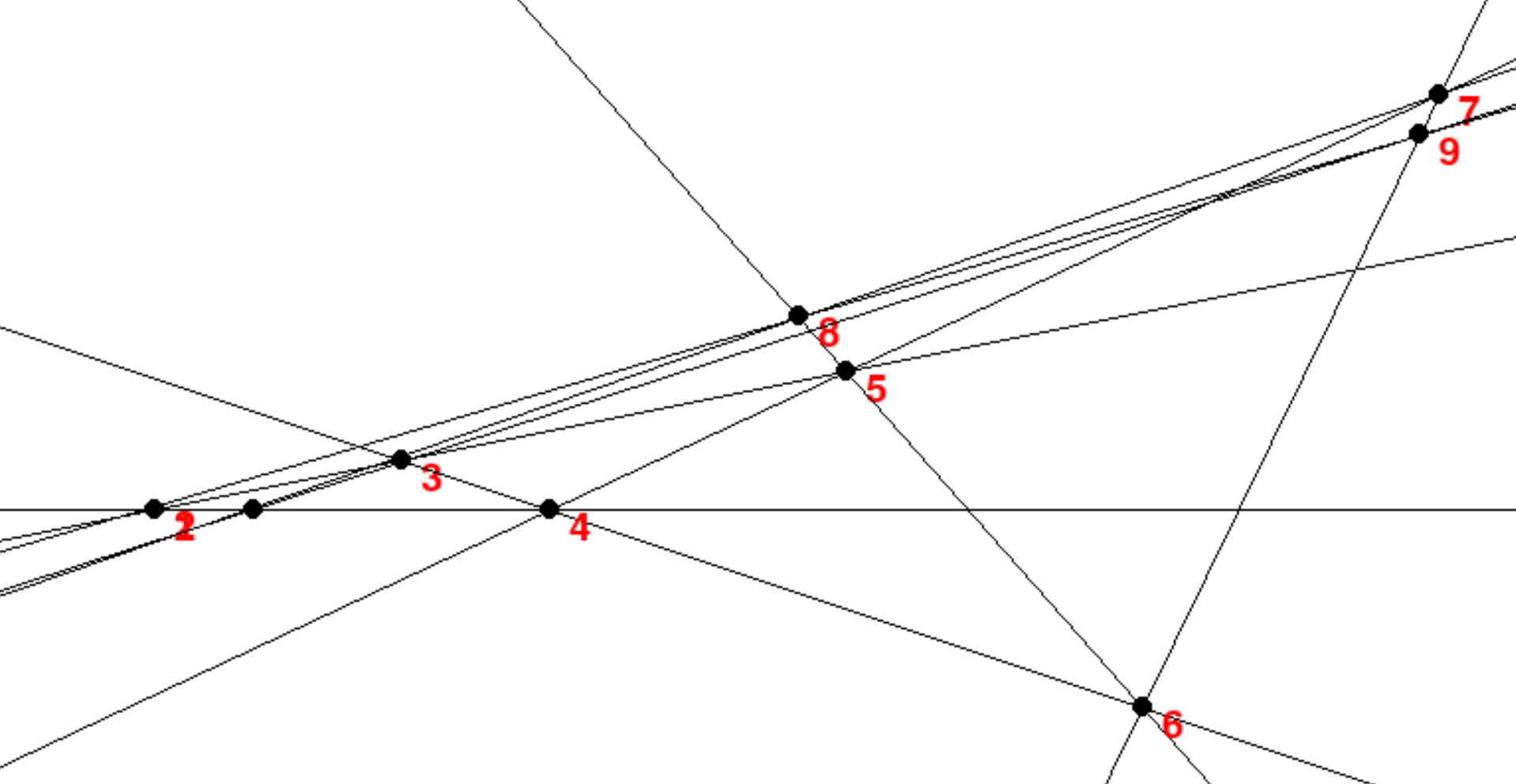}}\caption{}
\end{figure}
However, due to time constraints, there are many improvements and expansions that could be made.  We outline a few below.
\begin{itemize}
    \item For each $n$, the current way of generating the last three points, $n-1, n$ and $1$, is by randomly selecting positions on each of the appropriate lines, until each of the collinearity conditions are satisfied.  As the configuration increases in size, the probability of selecting these three points decreases exponentially, which is not ideal. This process needs significant refinement before large values of $n$ can be geometrically generated.
    \item As mentioned above, other algorithms could be implemented to generate other geometric realizations, such as those for $h$-astral configurations.
\end{itemize}

\medskip
\printbibliography

@book{Grunbaum,
  title={Configurations of Points and Lines},
  author={Grunbaum Branko},
  year={2009},
  publisher={American Mathematical Society}
}

@book{graphs,
  title={Configurations from a Graphical View Point},
  author={Pisanski, Servatius},
  year={2013},
  publisher={Birkhauser Advanced Texts}
}

@article{chiral,
  title={Chiral Astral Realizations of Cyclic 3-Configurations},
  author={Berman DeOrsey},
  year={2020},
  publisher={Springer},
}

@article{Brunel,
 author  = {Brunel, Georges},
 title   = {Polygones à autoinscription multiple},
 journal = {Procès-verbaux des Séances de la Société des Sciences Physiques et Naturelles de Bordeaux },
 year    = {1898},
 pages   = {43--46},
}

\end{document}